\documentclass{article}

\usepackage{amsmath}
\usepackage{amsthm}
\usepackage{amssymb}
\usepackage{delarray}
\usepackage[dvips]{graphics}
\usepackage{epsfig}
\usepackage{color}

\newtheorem{thm}{Theorem}[section]

\theoremstyle{definition}

\theoremstyle{remark}
\newtheorem*{obs}{Observation}

\title{An Interesting Number Theoretic Problem}
\author{Manjil P.~Saikia\\ Department of Mathematical Sciences,\\ Tezpur University,\\ Napaam, Dist. Sonitpur, India\\
\texttt{manjil.saikia@gmail.com}}

\begin{document}
\maketitle

\begin{abstract}
In this short note we prove a result that is an extension of an old Olympiad Problem and is a very simple variant of the question of finding an approximation for $k$ where it is a
nonzero constant and it satisfies the equation $a^{k}+b^{k}=c$ where $a,b,c$ are all real constants.
\end{abstract}

\section{Introduction and Motivation}

The motivation for this note and the result stated here is in the following question. Can we find an approximation for $k$ where it is a
nonzero constant and it satisfies the equation $a^{k}+b^{k}=c$ where $a,b,c$ are all real constants? We do not answer this question here, but we give a slightly different result, similar to the stated question partly inspired by an Olympiad problem.

\section{Main Result}

We state and prove the following extended theorem motivated by a discussion in \cite{ml}.

\begin{thm}Given a finite set $P$ of prime numbers, there exists a
positive integer $x$ such that it can be written in the form
$a^{p} + b^{p}$ ($a,b$ are positive integers), for each $p\in P$,
and cannot be written in that form for each $p$ not in $P$.
\end{thm}

\begin{proof}
Let $m$ denote the product of all primes $p$ which are in
$P$. Then we consider $x=2^{m+1}$.

Now we prove that $x$ satisfy the condition :
$x=2^{m}+2^{m}=(2^{\frac{m}{p}})^{p}+(2^{\frac{m}{p}})^{p}$,
$\forall p $ which are in $P$.

So, it can be represented in the form , $a^{p}+b^{p}$ $\forall p$
$\in P$. Next we will prove that the equation
,$2^{m}=a^{p}+b^{p}$, has no solution for which $p$ is not in $P$.

\textbf{Case-I:}~$p=2$, then $a=2^{k_{1}}a_{1}$ and
$b=2^{k_{2}}b_{1}$ where $a,b$ are congruent to $1$ modulo $2$. It
is easy to check that $k_{1}=k_{2}$ and $a_{1},b_{1}$ are
congruent to $1$ modulo $2$.

\textbf{Case-II:}~$a=2^{k_{1}}c,b=2^{k_{2}}d$ where $c,d$ are
congruent to $1$ modulo $2$. Suppose $a^{p}+b^{p}=2^{m+1}$ then
$k_{1}=k_{2}$. Because $\frac{c^{p}+d^{p}}{c+d}$ divides
$a^{p}+b^{p}$ and $\frac{c^{p}+d^{p}}{c+d}=1$ modulo $2$.
 from the above we can infer that $(c,d)=(1,1)$ and it gives that
 $a=b=2^{k}$. So $2^{pk}=2^{m}$ or we have that $p$ divides $m$
 and so $p$$\in P$.
 \end{proof}

\begin{obs}
Although we do not answer the original question but we can say that the answer to that question will be in the affirmative because we can certainly
find a set $A$, where we will find all the $k$'s.

Another aspect of the original problem will be to find approximations of $k$ such that instead of $c$ we have $c^{k}$. Fermat's Last Theorem certainly eliminates all the natural numbers from the set $A$ except $1,2$.

\end{obs}

\section{Acknowledgements}

The author would like to thank Dr.~Bhupen Deka for asking the original question, and to Prof.~Nayandeep Deka Baruah for reading through the very old draft \cite{mps} of this note.

\end{document}